\documentclass[a4paper,reqno,11pt]{amsart}
\usepackage{amssymb,amscd,amsxtra,psfrag,xypic}
\usepackage[dvipdfmx]{graphicx}
\usepackage{color}

\usepackage[pdftex]{hyperref}
\hypersetup{%
bookmarksnumbered=true,%
colorlinks=true,%
setpagesize=false,%
pdftitle={},%
pdfauthor={Takuzo Okada}}

\theoremstyle{plain}
\newtheorem{Thm}{Theorem}
\newtheorem{Lem}[Thm]{Lemma}

\newtheorem{Prop}[Thm]{Proposition}

\theoremstyle{definition}
\newtheorem{Def}[Thm]{Definition}
\newtheorem{Def-Lem}[Thm]{Definition-Lemma}

\newtheorem{Rem}[Thm]{Remark}

\newtheorem*{Ack}{Acknowledgments}

\theoremstyle{remark}

\newcommand{\prt}{\partial}

\newcommand{\Spec}{\operatorname{Spec}}

\newcommand{\Cox}{\operatorname{Cox}}
\newcommand{\bideg}{\operatorname{bi-deg}}

\newcommand{\mbA}{\mathbb{A}}

\newcommand{\mbC}{\mathbb{C}}

\newcommand{\mbG}{\mathbb{G}}

\newcommand{\mbP}{\mathbb{P}}

\newcommand{\mbZ}{\mathbb{Z}}

\newcommand{\mcI}{\mathcal{I}}

\newcommand{\mcL}{\mathcal{L}}
\newcommand{\mcM}{\mathcal{M}}
\newcommand{\mcN}{\mathcal{N}}
\newcommand{\mcO}{\mathcal{O}}

\newcommand{\mfm}{\mathfrak{m}}

\newcommand{\msp}{\mathsf{p}}
\newcommand{\msq}{\mathsf{q}}
\newcommand{\K}{\Bbbk}

\newcommand{\inj}{\hookrightarrow}

\newcommand{\ratmap}{\dashrightarrow}

\newcommand{\CH}{\operatorname{CH}}

\newcommand{\rest}{\operatorname{rest}}

\title[Hypersurfaces containing a linear space]{Stable rationality of index one Fano hypersurfaces containing a linear space}
\author{Takuzo~Okada}
\address{Department of Mathematics, Faculty of Science and Engineering, Saga University, Saga 840-8502 Japan}
\email{okada@cc.saga-u.ac.jp}
\date{}

\begin{document}

\begin{abstract}
We prove that a very general complex hypersurface of degree $n+1$ in $\mbP^{n+1}$ containing an $r$-plane with multiplicity $m$ is not stably rational under some mild assumptions for $n \ge 3, m, r > 0$. 
We also prove the failure of stable rationality of a very general hypersurface of degree $n+1$ in $\mbP^{n+1}$ admitting several isolated ordinary double points.
\end{abstract}

\maketitle


\section{Introduction} \label{sec:intro}

Throughout the introduction, the ground field is the complex number field $\mbC$.
By an {\it index one Fano hypersurface}, we mean a hypersurface of degree $n+1$ in $\mbP^{n+1}$.
It is proved by Totaro \cite{Totaro} that a very general smooth index one  Fano hypersurface is not stably rational for $n \ge 3$, and it is proved by de Fernex \cite{dF} that every smooth index one Fano hypersurface is not rational.
Here a very general hypersurface is a hypersurface corresponding to a complement of at most countable union of suitable proper closed subsets in the parameter space.
In general, given a family of algebraic varieties, special members of the family are closer to being (stably) rational compared to general members.
However, this is not always true and behavior of (stable) rationality in a family is subtle (see \cite{CT, dFF, HPT, NS, Perry, Tim, Totaro2}).

The aim of this paper is to study stable rationality of index one Fano hypersurfaces containing a linear space.
There are many results on the rationality of index one hypersurfaces of dimension $3$, i.e.\ quartic 3-folds.
Non-rationality of the following varieties are known: a very general quartic 3-fold containing a plane (\cite{Cheltsov}), a general quartic 3-fold containing a line doubly (\cite{CoMu}), and a factorial quartic 3-fold containing at most ordinary double points (\cite{Mella}).
In an arbitrary dimension, it is proved in \cite{Puk} that a general hypersurface of degree $n+1$ in $\mbP^{n+1}$ containing a unique point with multiplicity $n-1$ is not rational.
Apart from (index one) Fano hypersurfaces, it is worth while to mention that Schreieder \cite{Sch} proved that a very general hypersurface of degree $d$ in $\mbP^{n+1}$ containing an $r$-plane with multiplicity $d-2$ is not stably rational under some conditions on $n, d, r$.
Although we do not explain the above conditions, the inequality $d \ge n + 3$ must be satisfied, hence the above result does not cover (index one) Fano hypersurfaces. 
We state the main theorems of this paper.

\begin{Thm} \label{mainthm1}
Let $n \ge 3,m$ and $r$ be positive integers such that $n \ge m+r$.
Assume in addition that one of the following holds.
\begin{enumerate}
\item $m = n - 1$.
\item $n$ is even.
\item $n - m + 1$ is not a power of $2$.
\end{enumerate}
Then a very general complex hypersurface of degree $n+1$ in $\mbP^{n+1}$ containing an $r$-plane with multiplicity $m$ is not stably rational.
\end{Thm}

In the above statement, the hypersurface has at most canonical singularities, and it is terminal (resp.\ smooth) if and only if $n > m+r$ (resp.\ $m = 1$ and $r \le \lfloor n/2 \rfloor$).
For index one hypersurfaces containing isolated double points, we have the following.

\begin{Thm} \label{mainthm2}
Let $n \ge 3$ and $e$ be integers satisfying the following properties.
\begin{enumerate}
\item Either $n = 3$ or $n - 1$ is not a power  of $2$.
\item $e \le \max\{ \, l \in \mbZ \mid 2^l + l \le n \,\} + 1$.
\end{enumerate}
Let $\msp_1, \dots, \msp_e$ be distinct $e$ points of $\mbP^{n+1}$.
Then a very general complex hypersurface of degree $n+1$ in $\mbP^{n+1}$ containing $\msp_1, \dots, \msp_e$ as isolated ordinary double points is not stably rational.
\end{Thm}

\begin{Ack}
The author is partially supported by JSPS KAKENHI Grant Numbers 26800019 and 18K03216.
\end{Ack}


\section{Definition of families and reduction of the problem} \label{sec:deffam}

\subsection{Toric weighted projective space bundles}

We work over a field $\K$.
A {\it toric weighted projective space bundle over $\mbP^n$} is a projective simplicial toric variety $P$ with Cox ring
\[
\Cox (P) = \K [u, x] = \K [u_0,\dots,u_n,x_0,\dots,x_m],
\]
which is $\mbZ^2$-graded as
\[
\begin{pmatrix}
1 & \cdots & 1 & \lambda_0 & \cdots & \lambda_m \\
0 & \cdots & 0 & a_0 & \cdots & a_m
\end{pmatrix}
\]
and with the irrelevant ideal $I = (u_0,\dots,u_n) \cap (x_0,\dots,x_m)$, where $\lambda_0,\dots,\lambda_m$ are integers and $n, m, a_0,\dots,a_m$ are positive integers.
In other words, $P$ is the geometric quotient
\[
P = (\mbA^{n+m+2} \setminus V (I))/\mbG_m^2,
\]
where the action of $\mbG_m^2 = \mbG_m \times \mbG_m$ on $\mbA^{n+m+2} = \Spec \K [u,x]$ is given by the above matrix.
We will simply refer to $P$ as the WPS bundle over $\mbP^n$ defined by
\[
\begin{pmatrix}
u_0 & \cdots & u_n & & x_0 & \cdots & x_m \\
1 & \cdots & 1 & | & \lambda_0 & \cdots & \lambda_m \\
0 & \cdots & 0 & | & a_0 & \cdots & a_m
\end{pmatrix}.
\]
There is a natural projection $P \to \mbP^n$, which is the projection by the coordinates $u_0,\dots,u_n$, and its fiber is isomorphic to the weighted projective space $\mbP (a_0,\dots,a_m)$.

\begin{Rem}
Let $P$ be as above and assume that $a_0 = \cdots = a_m = 1$.
In this case $P$ is isomorphic to the $\mbP^m$-bundle
\[
P \cong \mbP_{\mbP^n} (\mcO_{\mbP^n} (-\lambda_0) \oplus \cdots \oplus \mcO_{\mbP^n} (-\lambda_m))
\]
over $\mbP^n$ and the morphism $P \to \mbP^n$ coincides with the projection.
\end{Rem}

Let $\msp \in P$ be a point and let $\msq \in \mbA^{n+m+2} \setminus V (I)$ be a preimage of $\msp$ via the morphism $\mbA^{n+m+2} \setminus V (I) \to P$ and write $\msq = (\alpha_0,\dots,\alpha_n,\beta_0,\dots,\beta_m)$.
In this case we express $\msp \in P$ as $\msp = (\alpha_0\!:\!\cdots\!:\!\alpha_n ; \beta_0\!:\!\cdots\!:\!\beta_n)$.
This is clearly independent of the choice of $\msq$.

Let $P$ be as above.
The Cox ring $\Cox (P)$ admits a natural grading by $\mbZ^2$ and we have the decomposition
\[
\Cox (P) = \bigoplus_{(\alpha,\beta) \in \mbZ^2} \Cox (P)_{(\alpha,\beta)},
\]
where $\Cox (P)_{(\alpha,\beta)} = \K [u,x]_{(\alpha,\beta)}$ consists of the homogeneous elements of bi-degree $(\alpha,\beta)$.
The Weil divisor class group $\operatorname{Cl} (P)$ of $P$ is isomorphic to $\mbZ^2$.
We denote by $\mcO_P (\alpha,\beta)$ the rank $1$ reflexive sheaf corresponding to the divisor class of type $(\alpha,\beta)$.
More generally, for a subscheme $Z \subset P$, we set $\mcO_Z (\alpha,\beta) = \mcO_X (\alpha,\beta)|_Z$.
Finally we remark that there is an isomorphism
\[
H^0 (P, \mcO_P (\alpha,\beta)) \cong \Cox (P)_{(\alpha,\beta)}.
\]

We give a description of standard open affine charts of $P$.
For $i = 0,\dots,n$ and $j = 0,\dots,m$, we define $U_{i,j} = (u_i \ne 0) \cap (x_j \ne 0) \subset P$.
Clearly the $U_{i,j}$ cover $P$.
We only explain an explicit description of $U_{i,j}$ for $j$ such that $a_j = 1$, which is enough for our purpose.
For $k \ne i$ and $l \ne j$, we set 
\[
\tilde{u}_k = \frac{u_k}{u_i} \quad \text{and} \quad \tilde{x}_l = \frac{u_i^{a_l \lambda_j - \lambda_l} x_l}{x_j^{a_l}},
\]
which are clearly $\mbG_m^2$-invariant rational functions on $P$ which are regular on $U_{i,j}$.
Moreover it is easy to see that $U_{i,j}$ is isomorphic to the affine $(n+m)$-space with affine coordinates $\{\tilde{u}_k \mid k \ne i\} \cup \{\tilde{x}_l \mid l \ne j\}$.

\begin{Rem} \label{rem:exrestmap}
Under the above setting, we note that the restriction map
\[
H^0 (P, \mcO_P (\alpha,\beta)) \cong \K [u,x]_{(\alpha,\beta)} \to H^0 (U_{i,j}, \mcO_{P} (\alpha,\beta)) \cong H^0 (\mbA^{n+m}, \mcO_{\mbA^{n+m}})
\]
can be understood as the homomorphism defined by sending $g (u,x) \in H^0 (P,\mcO_P (\alpha,\beta))$ to the polynomial obtained by substituting $\tilde{u}_k = \tilde{x}_l = 1$ in $g (\tilde{u},\tilde{x}) \in H^0 (U_{i,j}, \mcO_{P} (\alpha,\beta))$.
\end{Rem}

\subsection{Definition of families}

Let $n, m$ and $r$ be positive integers such that $n \ge m + r$.
We set $l := n + 1 - m$.
For a filed $\K$, we define $P_{\K} = P_{\K} (n,r)$ to be the WPS bundle over $\mbP^{n-r}$ defined by the $2 \times (n+3)$ matrix
\[
\begin{pmatrix}
u_0 & \cdots & u_{n-r} & & w & x_1 & \cdots & x_r & y \\
1 & \cdots & 1 & | & 0 & 0 & \cdots & 0 & -1 \\
0 & \cdots & 0 & | & 1 & 1 & \cdots & 1 & 1
\end{pmatrix}.
\]
Note that we have an isomorphism
\[
P_{\K} \cong \mbP_{\mbP^{n-r}} (\mcO_{\mbP^{n-r}}^{\oplus r+1} \oplus \mcO_{\mbP^{n-r}} (1)).
\]
We write 
\[
\K [u] = \K [u_0,\dots,u_{n-r}], \quad
\K [u,x,y] = \K [u_0,\dots,u_{n-r},x_1,\dots,x_r,y].
\]
For an integer $j$, we denote by $\K [u]_j$ the vector space consisting of the degree $j$ polynomials in variables $u_0,\dots,u_{n-r}$.
For integers $j_1, j_2$, we denote by $\K [u,x,y]_{(j_1,j_2)}$ the vector space consisting of the polynomials of bi-degree $(j_1,j_2)$, where the bi-degree is the one defined by the above action, i.e.\ $\bideg u_i = (1,0), \bideg (x_i) = (1,1)$ and $\bideg y = (0,1)$.
We set $\Lambda_{\K} = \Lambda_{\K} (n,m,r) := |\mcO_{P_{\K}} (m,l)|$ and let $\Xi_{\K} = \Xi_{\K} (n,m,r)$ be the linear subsystem of $\Lambda_{\K}$ spanned by the sections
\[
\{ \, a_m w^l + f \mid a_m \in \K [u]_m, f \in \K [u,x,y]_{(m,l)} \, \}.
\]

The complete linear system $|\mcO_{P_{\K}} (0,1)|$ is base point free and $H^0 (P_{\K}, \mcO_P (0,1))$ is spanned by $w,x_1,\dots,x_r,u_0 y,\dots,u_{n-r} y$.
Let $\Psi \colon P_{\K} \to \mbP^{n+1}_{\K}$ be the the morphism associated to $|\mcO_{P_{\K}} (0,1)|$.
We denote by $w,x_1,\dots,x_r,y_0,\dots,y_{n-r}$ the homogeneous coordinates of $\mbP^{n+1}$ so that $\Psi^*w = w$, $\Psi^*x_i = x_i$ and $\Psi^* y_i = u_i y$. 
It is easy to see that $\Psi$ is the blowup of $\mbP^{n+1}_{\K}$ along the linear space $L _{\K}= (y_0 = \cdots = y_{n-r} = 0) \subset \mbP^{n+1}_{\K}$ of dimension $r$ and its exceptional divisor $E_{\K}$ is the divisor $(y = 0) \subset P_{\K}$, which is isomorphic to $\mbP^{n-r}_{\K} \times \mbP^r_{\K}$.
The push-forward $\Psi_*$ of divisors induces a one-to-one correspondence between $\Lambda_{\K}$ and the linear system $|\mcI_{L_{\K}}^m \mcO_{\mbP^{n+1}_{\K}} (n+1)|$ of degree $n+1$ hypersurfaces in $\mbP^{n+1}_{\K}$ containing $L_{\K}$ with multiplicity at least $m$.

For the proof of Theorem \ref{mainthm1}, it is enough to show the failure of stable rationality of a very general member of $\Lambda_{\mbC}$.

\begin{Lem} \label{lem:smchar0Xi}
Suppose that the ground field $\K$ is an algebraically closed field of characteristic $0$.
Then general members of $\Lambda_{\K}$ and $\Xi_{\K}$ are smooth.
\end{Lem}

\begin{proof}
It is easy to see that $\Lambda_{\K}$ and $\Xi_{\K}$ are base point free.
In fact, $\Lambda_{\K}$ contains $\Xi_{\K}$ as a linear subsystem and $\Xi_{\K}$ contains the linear subsystem generated by the sections 
\[
\{ u_i^m w^l, u_i^m x_j^l, u_i^{n+1} y^l \mid 0 \le i \le n-r, 1 \le j \le r\},
\]
which is clearly base point free. 
Thus the assertion follows from Bertini theorem since $P_{\K}$ is smooth.
\end{proof}

\begin{Rem} \label{rem:known}
Theorem \ref{mainthm1} is known when $(m,r) = (n-1,1), (n-2,2)$ as we will explain below.

Suppose that $(m,r) = (n-1,1)$.
Then a general member $X \in \Lambda_{\mbC} (n,n-1,1)$, together with the projection $X \to \mbP^{n-1}$, is a conic bundle (embedded in $P_{\mbC} (n,n-1,1) \cong \mbP_{\mbP^{n-1}} (\mcO \oplus \mcO \oplus \mcO (1))$) such that $\mcO_X (-K_X) \cong \mcO_X (0,1)$ is not ample (in fact $\left| -K_X \right|$ defines the blowup $\Psi|_X \colon X \to \Psi_*X$ along the line $L_{\mbC}$).
By \cite[Theorem 1.1]{AO}, a very general $X$ is not stably rational. 

Suppose that $(m,r) = (n-2,2)$.
Then a general member $X \in \Lambda_{\mbC} (n,n-2,2)$, together with the projection $X \to \mbP^{n-2}$, is a del Pezzo fibration of degree $3$ (embedded in $P_{\mbC} (n,n-2,2) \cong \mbP_{\mbP^{n-2}} (\mcO^{\oplus 3} \oplus \mcO (1))$) such that $\mcO_X (-K_X) \cong \mcO_X (0,1)$ is not ample.
By \cite[Theorem 1.1]{KO}, a very general $X$ is not stably rational.
\end{Rem}

\subsection{Reduction of the problem} \label{subsec:reduction}

We explain that Theorem \ref{mainthm1} is reduced to the following.

\begin{Prop} \label{mainprop}
Let $n \ge 3$, $m$, $r$ be positive integers such that $n \ge m+ r$, and let $p$ be a prime number which divides $l := n - m + 1$.
We assume that $p \ne 2$ if $l \ne 2$ and $n$ is odd. 
Let $\K$ be an algebraically closed field of characteristic $p$.
Then a very general member $X \in \Xi_{\K}$ admits a universally $\CH_0$-trivial resolution $\varphi \colon \tilde{X} \to X$ such that $H^0 (\tilde{X}, \Omega_{\tilde{X}}^{n-1}) \ne 0$.
\end{Prop}

We recall  the definition and a basic property concerning universal $\CH_0$-triviality.
For a variety $X$, we denote by $\CH_0 (X)$ the {\it Chow group} of $0$-cycles on $X$, which is the free abelian group of $0$-cycles modulo rational equivalence.

\begin{Def}
A projective variety $X$ defined over a field $k$ is {\it universally $\CH_0$-trivial} if, for any field extension $F \supset k$, the degree map $\deg \colon \CH_0 (X_F) \to \mbZ$ is an isomorphism.
A projective morphism $\varphi \colon Y \to X$ defined over a field $k$ is {\it universally $\CH_0$-trivial} if, for any field extension $F \supset k$, the pushforward map $\varphi_* \colon \CH_0 (Y_F) \to \CH_0 (X_F)$ is an isomorphism.
\end{Def}

\begin{Lem}
If $X$ is a smooth, projective, stably rational variety, then $X$ is universally $\CH_0$-trivial.
\end{Lem}

We explain that Theorem \ref{mainthm1} follows from Proposition \ref{mainprop}.
Let $n, m$ and $r$ be as in Theorem \ref{mainthm1}.
For the proof of Theorem \ref{mainthm1}, it is enough to show that a very general $V \in \Lambda_{\mbC} (n,m,r)$ is not universally $\CH_0$-trivial.
We can degenerate $V$ to a very general member $V'$ of $\Xi_{\mbC} (n,m,r)$.
Since $V$ and $V'$ are smooth by Lemma \ref{lem:smchar0Xi}, we can apply the specialization theorem \cite[Theorem 2.1]{Voisin} of universal $\CH_0$-triviality and it is then enough to show that $V'$ is not universally $\CH_0$-trivial.
Let $p$ be a prime number dividing $l := n - m + 1 \ge 2$.
If $l \ne 2$ and $n$ is odd, then $l$ is not a power of $2$ by the assumption of Theorem \ref{mainthm1}, and we can assume $p \ne 2$.
Hence $n, m, r$ and $p$ satisfy the assumptions of Proposition \ref{mainprop}.
Let $\K$ be an algebraically closed field of characteristic $p$, and let $X$ be a very general member of $\Xi_{\K} (n,m,r)$.
We can lift $X$ to $V'$ (using the ring of Witt vectors with coefficient $\K$, which is a mixed characteristic DVR with residue field $\K$).
Now we assume that Proposition \ref{mainprop} holds and let $\varphi \colon \tilde{X} \to X$ be as in Proposition \ref{mainprop}.
By \cite[Lemma 2.2]{Totaro}, $\tilde{X}$ is not universally $\CH_0$-trivial.
Then, applying the specialization theorem \cite[Th\'eor\`eme 1.14]{CTP} of universal $\CH_0$-triviality, we conclude that Theorem \ref{mainthm1} follows from Proposition \ref{mainprop}.

\section{Proof of Theorem \ref{mainthm1}} \label{sec:proof1}

The aim of this section is to prove Proposition \ref{mainprop} from which Theorem \ref{mainthm1} follows.

\subsection{Construction of global differential forms}

Before starting the proof of Proposition \ref{mainprop}, we recall the construction of global differential forms on an inseparable covering space.
We refer readers to \cite[Chapter V.5]{Kollar} and \cite{Okcyclic} for details.

Let $Z$ be a smooth variety defined over an algebraically closed field $\K$ of characteristic $p > 0$, $\mcL$ an invertible sheaf on $Z$, $m$ a positive integer divisible by $p$ and $s \in H^0 (Z, \mcL^m)$.
Let 
\[
\pi_U \colon U := \Spec (\oplus_{i \ge 0} \mcL^{-i}) \to Z
\] 
be the total space of the line bundle $\mcL$.
We denote by $y \in H^0 (U, \pi_U^* \mcL)$ the zero section and define
\[
Z [\sqrt[m]{s}] := (y^m - \pi_U^* s = 0) \subset U.
\]
Set $X = Z [\sqrt[m]{s}]$ and $\pi = \pi_U|_X \colon X \to Z$.
We call $X$ or $\pi \colon X \to Z$ the {\it covering of $Z$ obtained by taking the $m$th roots of $s$}.

The singularities of $X$ can be analyzed by critical points of the section $s$.
Let $\msq \in Z$ be a point and $x_1,\dots,x_n$ local coordinates of $Z$ at $\msq$.
Around $\msq$, we can write $s = f (x_1,\dots,x_n) \tau^m$, where $f \in \mcO_{Z,\msq}$ and $\tau$ is a local generator of $\mcL$ at $\msq$.
We write 
\[
f = \alpha + \ell + q + g,
\] 
where $\alpha \in \K$ and $\ell, q$ are linear, quadratic forms in $x_1,\dots,x_n$, respectively, and $g = g (x_1,\dots,x_n) \in (x_1,\dots,x_n)^3$.

\begin{Def}
Let $s \in H^0 (Z, \mcL^m)$ be a section and we keep the above setting.
We say that $s$ has a {\it critical point} at $\msq \in Z$ if $\ell = 0$.
We say that $s$ has a {\it nondegenerate critical point} at $\msq \in Z$ if $s$ has a critical point at $\msq$ and the following properties are satisfied.
\begin{itemize}
\item Either $p \ne 2$ or $p = 2$ and $n$ is even.
\item $q$ is a nondegenerate quadric.
\end{itemize}
\end{Def}

Note that nondegenerate critical points are isolated.
It is easy to see that $X$ is singular at $\msp \in X$ if and only if $s$ has a critical point at $\pi (\msp)$.
Thus, if the section $s$ has only nondegenerate critical points on $Z$, then the singularity of $X$ are isolated.

We can summarize the results of \cite{Kollar}, \cite{CTPcyclic} and \cite{Okcyclic} in the following form.

\begin{Lem}[{\cite[Chapter V.5]{Kollar}, \cite{CTPcyclic}, \cite[Proposition 4.1]{Okcyclic}}] \label{lem:covtech}
Let $X, Z, \mcL, m$ and $s$ be as above.
Assume that $s \in H^0 (Z, \mcL^m)$ has only nondegenerate critical points on $Z$.
Then there exists an invertible subsheaf $\mcM$ of the double dual $(\Omega_X^{n-1})^{\vee \vee}$ of the sheaf $\Omega_X^{n-1}$ and a resolution of singularities $\varphi \colon \tilde{X} \to X$ with the following properties.
\begin{enumerate}
\item $\mcM \cong \pi^* (\omega_Z \otimes \mcL^m)$.
\item $\varphi$ is universally $\CH_0$-trivial and $\varphi^* \mcM \inj \Omega_{\tilde{X}}^{n-1}$
\end{enumerate}
\end{Lem}

We will refer to $\mcM$ in the above lemma as the {\it invertible subsheaf of $(\Omega_X^{n-1})^{\vee \vee}$ associated to the covering $\pi \colon X \to Z$}. 

\subsection{Proof of Proposition \ref{mainprop}}

Let $n, m, r$ and $p$ be as in the statement of Proposition \ref{mainprop}.
We set $l := n - m + 1$.
Recall that $p \ne 2$ when $m \ne n-1$ and $n$ is odd.
We freely use notation of Section \ref{sec:deffam}.

By Remark \ref{rem:known}, the proof is done when  $l = 2$, or equivalently when $(m,r) = (n-1,1)$.
In the following, we assume $l \ge 3$.  
Throughout this section, we work over an algebraically closed field $\K$ of characteristic $p$.

Let $X \in \Xi_{\K} = \Xi_{\K} (n,m,r)$ be a very general member with defining equation $a_m w^l + f = 0$.
The aim is to show that $X$ admits a universally $\CH_0$-trivial resolution $\varphi \colon \tilde{X} \to X$ of singularities such that $\tilde{X}$ is not universally $\CH_0$-trivial.
Let $Q = Q_{\K}$ be the WPS bundle over $\mbP^{n-r}$ defined by
\[
\begin{pmatrix}
u_0 & \cdots & u_{n-r} & & z & x_1 & \cdots & x_r & y \\
1 & \cdots & 1 & | & 0 & 0 & \cdots & 0 & -1 \\
0 & \cdots & 0 & | & l & 1 & \cdots & 1 & 1
\end{pmatrix}
\]
and define 
\[
Z := (a_m z + f = 0) \subset Q.
\]
We set 
\[
\begin{split}
\Delta &:= (x_1 = \cdots = x_r = y = 0) \subset Q, \\
\Delta_Z &:= \Delta \cap Z = (a_m = x_1 = \cdots = x_r = y = 0) \subset Z, \\
Z^{\circ} &:= Z \setminus \Delta_Z = Z \setminus (a_m = x_1 = \cdots = x_r = y = 0) \subset Z, \\ 
X^{\circ} &:= \pi^{-1} (Z^{\circ}) = X \setminus (a_m = x_1 = \cdots = x_r = y = 0) \subset X.
\end{split}
\]
We denote by $\pi \colon X \to Z$ the morphism defined by $\pi^*z = w^l$, which is the restriction of the natural morphism $P \to Q$.

\begin{Lem} \label{lem:nonsingcharp}
$Z^{\circ}$ is smooth, and $X$ is smooth along $X \setminus X^{\circ}$.
\end{Lem}

\begin{proof}
We have 
\[
X \setminus X^{\circ} = (a_m = x_1 = \cdots = x_r = y = 0) \subset P.
\]
Since $a_m$ is general, we may assume that the hypersurface $a_m = 0$ in $\mbP^{n-r}$ is smooth.
This implies that $a_m w^l$ vanishes at any point of $X \setminus X^{\circ}$ with multiplicity at most $1$ and thus $X$ is nonsingular at any point of $X \setminus X^{\circ}$.

We claim that the restriction map
\[
H^0 (Q, \mcO_Q (m,l)) \to \mcO_Q (m,l) \otimes (\mcO_{Q,\msq}/\mfm_{\msq}^2)
\]
is surjective for any point $\msq \in Q^{\circ} := Q \setminus \Delta$.
Note that $Q^{\circ}$ is the smooth locus of $Q$.
We set 
\[
\begin{split} 
U_{i,j} &= (u_i \ne 0) \cap (x_j \ne 0) \subset Q, \text{ for $i = 0,\dots,n-r$, $j = 1,\dots,r$}, \\
U_{i,y} &= (u_i \ne 0) \cap (y \ne 0) \subset Q, \text{ for $i = 0,\dots,n-r$},
\end{split}
\] 
so that $Q^{\circ}$ is covered by the $U_{i,j}$ and the $U_{i,y}$.
Suppose that $\msq \in U_{i,j}$.
Without loss of generality, we may assume $\msq \in U_{0,1}$.
We have an isomorphism
\[
U_{0,1} \cong \mbA^{n+1} = \Spec \K [\tilde{u}_1,\dots,\tilde{u}_{n-r},\tilde{z},\tilde{x}_2,\dots,\tilde{x}_r,\tilde{y}],
\]
where
\[
\tilde{u}_i = \frac{u_i}{u_0}, \ 
\tilde{z} = \frac{z}{x_1^l}, \ 
\tilde{x}_j = \frac{x_j}{x_1}, \ 
\tilde{y} = \frac{y u_0}{x_1}.
\]
There are sections
\[
u_0^m x_1^l, \ 
u_i u_0^{m-1} x_1^l, \ u_0^m z, \ u_0^m x_j x_1^{l-1}, \ u_0^{m+1} y x_1^{l-1} \in H^0 (Q, \mcO_Q (m,l)), 
\]
and their restriction to $U_{0,1}$ are as follows
\[
1, \ 
\tilde{u}_i, \ \tilde{z}, \ \tilde{x}_j, \ \tilde{y}.
\]
This shows that the restriction map is surjective for $\msq \in U_{0,1}$, and thus for any $\msq \in U_{i,j}$.
Suppose that $\msq \in U_{i,y}$.
We may assume $\msq \in U_{0,y}$.
We have an isomorphism
\[
U_{0,y} \cong \mbA^{n+1} = \Spec \K [\tilde{u}_1,\dots,\tilde{u}_{n-r},\tilde{z},\tilde{x}_1,\dots,\tilde{x}_r],
\]
where
\[
\tilde{u}_i = \frac{u_i}{u_0}, \ 
\tilde{z} = \frac{z}{u_0^l y^l}, \ 
\tilde{x}_j = \frac{x_j}{u_0 y}.
\]
There are sections
\[
u_0^{m+l} y^l, \ u_i u_0^{m+l-1} y^l, \ u_0^m z, \ u_0^{m+l-1} x_j y^{l-1} \in H^0 (Q, \mcO_Q (m,l))
\]
and their restriction to $U_{0,y}$ are as follows
\[
1, \ \tilde{u}_i, \ \tilde{z}, \ \tilde{x}_j.
\] 
This shows that the restriction map is surjective for $\msq \in U_{0,y}$, and thus for any $\msq \in U_{i,y}$.
Therefore the claim is proved.

Note that $Z$ is a general member of the linear system $|\mcO_Q (m,l)|$.
By the surjectivity of the above restriction map, it follows that, for a given point $\msq \in Q^{\circ}$, $n+2$ independent conditions are imposed for members of $|\mcO_Q (m,l)|$ to have a singular point at $\msq$. 
By the dimension counting argument, we can conclude that a general member of $|\mcO_Q (m,l)|$ is smooth outside $\Delta$, and the proof is completed.
\end{proof}

We set $\mcL = \mcO_Z (0,1)$ and $\mcL^{\circ} = \mcL|_{Z^{\circ}}$.
Note that $z$ can be viewed as a global section of $\mcL^l = \mcL^{\otimes l}$.
We see that $\pi^{\circ} = \pi|_{X^{\circ}} \colon X^{\circ} \to Z^{\circ}$ is the morphism obtained by taking the $p$th roots of $z \in H^0 (Z^{\circ}, \mcL^l)$.
In the following we choose and fix a general $a = a_m \in \K [u]_m$.
Then, we will show that the section $z \in H^0 (Z, \mcL^l)$ has only admissible critical points on $Z^{\circ}$ for a general $f \in \K [x,y]_{(m,l)}$.

Let $R$ be the WPS bundle over $\mbP^{n-r}$ defined by
\[
\begin{pmatrix}
u_0 & \cdots & u_{n-r} & & x_1 & \cdots & x_r & y \\
1 & \cdots & 1 & | & 0 & \cdots & 0 & -1 \\
0 & \cdots & 0 & | & 1 & \cdots & 1 & 1
\end{pmatrix}.
\]
We have a natural projection $Q \ratmap R$ which is defined outside $\Delta \subset Q$.
Then its restriction $Z \ratmap R$ to $Z$ is defined outside $\Delta_Z$ and we denote by $\rho \colon Z^{\circ} \to R$ the restriction of $Q \ratmap R$ on $Z^{\circ} = Z \setminus \Delta_Z$.

\begin{Lem} \label{lem:critprelim}
The section $z \in H^0 (Z, \mcL^l)$ does not have a critical point along $(a = 0) \cap Z^{\circ}$.
Moreover, $z$ has a nondegenerate critical point at $\msq \in (a \ne 0) \cap Z^{\circ}$ if and only if the section $a^{p-1} f \in H^0 (R, \mcO_R (p m,l))$ has a nondegenerate critical point at $\rho (\msq) \in (a \ne 0) \cap R$.
\end{Lem}

\begin{proof}
Let $\msq \in (a = 0) \cap Z^{\circ}$.
Then, since $\prt (a z + f)/\prt z = a$ and $Z^{\circ}$ is non-singular, $z$ (or its translation) can be chosen as a part of local coordinates of $Z^{\circ}$ at $\msq$ and this clearly implies that $z$ does not have a critical point at $\msq$.
This proves the first assertion.

Let $\msq \in (a \ne 0) \cap Z^{\circ}$.
Then, since $a$ does not vanish at $\msq$, we see that $z$ has an admissible critical at $\msq$ if and only if so does $a^p z$.
By the defining equation of $Z$, we have $- a^p z = a^{p-1} f$.
It follows that $z$ has an admissible critical point at $\msq$ if and only if so does  $a^{p-1} f \in H^0 (Z, \mcO_Z (p m, l))$.
It is now easy to see that the latter is equivalent to saying that the section $a^{p-1} f$ viewed as a section of $\mcO_R (p m,l))$ has an admissible critical point at $\rho (\msq) \in (a \ne 0) \cap R$.
This completes the proof. 
\end{proof}

For a positive integer $k$ and a point $\msq \in R$, we define
\[
\rest^k_{\msq} \colon H^0 (R,\mcO_R (m, l)) \to \mcO_R (m,l) \otimes (\mcO_{R,\msq}/\mfm^k_{\msq}).
\]
We set
\[
\begin{split}
V_{i,j} &:= (u_i \ne 0) \cap (x_j \ne 0) \subset R, \text{ for $i = 0,\dots,n-r, j = 1,\dots,r$}, \\
V_{i,y} &:= (u_i \ne 0) \cap (y \ne 0) \subset R, \text{ for $i = 0,\dots,n-r$}, \\
V_y &:= \bigcup_{0 \le i \le n-r} V_{i,y} = (y \ne 0) \subset R, \\
\Gamma &:= (y = 0) \subset R.
\end{split}
\]

\begin{Lem} \label{lem:restR}
\begin{enumerate}
\item $\rest^2_{\msq}$ is surjective for any $\msq \in \Gamma$.
\item $\rest^3_{\msq}$ is surjective for any $\msq \in V_y$.
\end{enumerate}
\end{Lem}

\begin{proof}
Let $\msq \in \Gamma$ be a point.
Since $\Gamma$ is covered by the $U_{i,j}$, we may assume, by replacing coordinates $u, x$, that $\msq = (1\!:\!0\!:\!\cdots\!:\!0; 1\!:\!\cdots\!:\!0\!:\!0) \in V_{0,1}$.
We have an isomorphism
\[
V_{0,1} \cong \mbA^n = \Spec \K [\tilde{u}_1,\dots,\tilde{u}_{n-r},\tilde{x}_2,\dots,\tilde{x}_r,\tilde{y}],
\]
where
\[
\tilde{u}_i = \frac{u_i}{u_0}, \ 
\tilde{x}_j = \frac{x_j}{x_1}, \ 
\tilde{y} = \frac{y u_0}{x_1},
\]
and $\tilde{u}_1,\dots,\tilde{u}_{n-r},\tilde{x}_2,\dots,\tilde{x}_r,\tilde{y}$ can be choosen as local coordinates of $R$ at $\msq$. 
The restriction map is simply described as
\[
\rest^2_{\msq} (g (u,x,y)) = g (1,\tilde{u}_1,\dots, \tilde{u}_{n-r},1,\tilde{x}_2,\dots,\tilde{x}_r,\tilde{y}) \pmod{\mfm^2_{\msq}}
\]
for $g = g (u,x,y) \in \K [u,x,y]_{(m,l)}$.
Since $l \ge 1$ (in fact $l \ge 3$) and $m \ge 1$, we have sections
\[
u_0^m x_1^l, \ 
u_0^{m-1} u_i x_r^l, \ 
u_0^m x_i x_1^{l-1}, \ 
u_0^{m+1} y x_1^{l-1} \in H^0 (R, \mcO_R (m,l))
\]
and their restriction to $V_{0,1}$ are $1,\tilde{u}_i,\tilde{x}_i,\tilde{y}$, respectively.
Thus $\rest^2_{\msq}$ is surjective and (1) is proved.

Let $\msq \in V_y$ be a point.
Replacing coordinates, we may assume $\msq = (1\!:\!0\!:\!\cdots\!:\!0;  0\!:\!\cdots\!:\!0\!:\!1) \in V_{0,y}$. 
We have an isomorphism
\[
V_{0,y} \cong \mbA^n = \K [\tilde{u}_1,\dots,\tilde{u}_{n-r},\tilde{x}_1,\dots,\tilde{x}_r],
\]
where
\[
\tilde{u}_i = \frac{u_i}{u_0}, \ 
\tilde{x}_j = \frac{x_j}{y u_0}.
\]
In this case $\tilde{u}_1,\dots,\tilde{u}_{n-r},\tilde{x}_1,\dots,\tilde{x}_r$ can be chosen as local coordinates of $R$ at $\msq$ and the restriction map is described as
\[
\rest^4_{\msq} (g (u,x,y)) = g (1,\tilde{u}_1,\dots,\tilde{u}_{n-r},\tilde{x}_1,\dots,\tilde{x}_r,1) \pmod{\mfm^4_{\msq}}
\]
for $g = g (u,x,y) \in \K [u,x,y]_{(m,l)}$.
Since $l \ge 2$, there are sections
\[
u_0^{n+1} y^l, \ u_0^n u_i y^l, \ u_0^{n+1} x_j y^{l-1}, \ 
u_0^{n-1} u_{i_1} u_{i_2} y^l, \ u_0^n u_i x_j y^{l-1}, \ u_0^{n+1} x_{j_1} x_{j_2} y^{l-2},
\]
of bi-degree $(m,l)$ whose image via $\rest^3_{\msq}$ form a basis of $\mcO_R/\mfm_{\msq}^3$.
This proves (2).
\end{proof}

\begin{Lem} \label{lem:admcrit}
The section $z \in H^0 (Z^{\circ}, \mcL^l)$ has only nondegenerate critical points on $Z^{\circ}$.
\end{Lem}

\begin{proof}
By Lemma \ref{lem:critprelim}, it is enough to show that the section $a^{p-1} f \in H^0 (R, \mcN)$ has only nendegenerate critical points on $V_a := (a \ne 0) \subset R$.

We set
\[
W = \{\, a^{p-1} g \mid g = g (u,x,y) \in \K [u,x,y]_{(m,l)} \,\},
\]
which we identify with a subspace of $H^0 (R,\mcO_R (pm,l))$.
For a point $\msq \in V_a$ and a positive integer $k$, we define
\[
\rest^k_{W,\msq} \colon W \to \mcO_R (pm,l) \otimes (\mcO_{R,\msq} /\mfm^k_{\msq}).
\]
For an element $a^{p-1} g \in W$, we have $g \in H^0 (R, \mcO_R (l,m))$ and 
\[
\rest^k_{W,\msq} (a^{p-1} g) = \bar{a}_{\msq}^{p-1} \rest^k_{\msq} (g),
\]
where $\bar{a}_{\msq}$ is the restriction of $a \in H^0 (R,\mcO_R (m,0))$ to $\mcO_R (m,0) \otimes (\mcO_{R,\msq}/\mfm_{\msq}^k)$.
Since $a$ does not vanish at $\msq \in V_a$, we see that surjectivity of $\rest^k_{\msq}$ implies that of $\rest^k_{W,\msq}$. 

By Lemma \ref{lem:restR}, $\rest^3_{\msq}$ is surjective for any $\msq \in V_a \cap V_y$, hence so is $\rest^3_{W,\msq}$.
This means that a general $a^{p-1} f \in W$ has only nondegenerate critical points on $V_a \cap V_y$.
Again by Lemma \ref{lem:restR}, $\rest^2_{\msq}$ is surjective for any $\msq \in V_a \cap \Gamma$, hence so is $\rest^2_{W,\msq}$.
Surjectivity of $\rest^2_{W,\msq}$ implies that $n = \dim R$ independent conditions are imposed for elements of $W$ to have a critical point at $\msq \in V_a \cap \Gamma$.
Since $\dim \Gamma = n-1$, we conclude, by counting dimensions, that a general $a^{p-1} f \in W$ does not have a critical point along $V_a \cap \Gamma$.
This completes the proof.
\end{proof}

\begin{proof}[Proof of \emph{Proposition \ref{mainprop}} and \emph{Theorem \ref{mainthm1}}]
Let $X \in \Xi_{\K}$ be as above.
Let $\mcM^{\circ}$ be the invertible subsheaf of $(\Omega^{n-1}_{X^{\circ}})^{\vee \vee}$ associated to $\pi^{\circ} = \pi|_{X^{\circ}} \colon X^{\circ} \to Z^{\circ}$ and let $\mcM$ be the pushforward of $\mcM^{\circ}$ via the open immersion $X^{\circ} \inj X$.
We set $\lambda = n-m-r$.
Then we have
\[
\mcM^{\circ} \cong {\pi^{\circ}}^* (\omega_{Z^{\circ}} \otimes \mcL^l) \cong {\pi^{\circ}}^* \mcO_{Z^{\circ}} (-\lambda,\lambda) \cong \mcO_{X^{\circ}} (-\lambda,\lambda),
\]
so that $\mcM \cong \mcO_X (-\lambda,\lambda)$.
Thus we have $H^0 (X, \mcM) \ne 0$ since $\lambda \ge 0$.
Note that, since $X \setminus X^{\circ}$ is smooth by Lemma \ref{lem:nonsingcharp}, the singularity of $X$ are all contained in $X^{\circ}$ and is coming from nondegenerate critical points of $z \in H^0 (Z^{\circ}, \mcL^l)$. 
By Lemmas \ref{lem:admcrit} and \ref{lem:covtech}, there exists a universally $\CH_0$-trivial resolution $\varphi \colon \tilde{X} \to X$ such that $\varphi^*\mcM \inj \Omega^{n-1}_{\tilde{X}}$.
Then we have $H^0 (\tilde{X}, \Omega^{n-1}_{\tilde{X}}) \ne 0$ and the proof of Proposition \ref{mainprop} is completed.
As explained in Section \ref{subsec:reduction}, Theorem \ref{mainthm1} follows from Proposition \ref{mainprop}.
\end{proof}

\section{Proof of Theorem \ref{mainthm2}}

This section is devoted to the proof of Theorem \ref{mainthm2}.
We work over $\mbC$.
Let $n \ge 3$ and $e$ be as in the statement of Theorem \ref{mainthm2}.
Let $\msp_1,\dots,\msp_e$ be distinct points of $\mbP^{n+1} = \mbP^{n+1}_{\mbC}$ and let $V$ be a very general hypersurface of degree $n+1$ in $\mbP^{n+1}$ containing $\msp_i$ with multiplicity $2$ for $i = 1,\dots, e$, that is, $V \in |\mcI^2_{\msp_1} \cdots \mcI^2_{\msp_e} \mcO_{\mbP^{n+1}} (n+1)|$.
We will show that $V$ is not stably rational.

We set $r = \max \{\, l \in \mbZ \mid 2^r + r \le n \, \}$.
Note that $e \le r + 1$ and $r \ge 1$.
We can choose and fix an $r$-plane $L \subset \mbP^{n+1}$ containing $\msp_1,\dots,\msp_e$.
Then we can degenerate $V$ to a very general hypersurface $V'$ of degree $n+1$ in $\mbP^{n+1}$ containing $L$ with multiplicity $2$.
Let $\psi \colon X \to V'$ be the blowup of $V'$ along $L$ so that $X$ is a very general member of $\Lambda_{\mbC} (n,2,r)$ and $\psi$ is a resolution of singularities.
We have $n \ge 2^r + r \ge 2 + r$.
Let $p$ be a prime number dividing $l := n - 2 + 1 = n - 1 \ge 2$.
We assume $p \ne 2$ when $n \ne 3$, which is possible since $n - 1$ is not a power of $2$ in this case.
We can degenerate $X$ to a very general member $X' \in \Lambda_{\K} (n, 2, r)$, where $\K$ is an algebraically closed field of characteristic $p$. 
By Proposition \ref{mainprop} and the specialization theorem \cite[Th\'eor\`eme 1.14]{CTP}, $X$ is not universally $\CH_0$-trivial.
Hence, again by the specialization theorem \cite[Th\'eor\`eme 1.14]{CTP}, the proof will be completed if we show that  $\psi$ is universally $\CH_0$-trivial.

The $\psi$-exceptional divisor $E$ is $(y = 0) \cap X$, which is isomorphic to a very general hypersurface of bi-degree $(2,n-1)$ in $\mbP^{n-r} \times \mbP^r$.
Moreover, the restriction $\psi|_E \colon E \to L$ coincides with the restriction of the second projection $\mbP^{n-r} \times \mbP^r \to \mbP^r \cong L$ to $E$.
It follows that $\psi|_E \colon E \to L$ is a quadric bundle over $L$.
Let $\xi \in L$ be any scheme point of $Y$ and let $E_{\xi} = \psi^{-1} (\xi)$ be the fiber considered as a variety over the residue field $k (\xi)$ of $V'$ at $\xi$, which is a smooth quadric in $\mbP^{n-r}_{k (\xi)}$.
The transcendental degree of $k (\xi)$ over $\mbC$ is at most $r$.
Thus, by \cite[Theorem 2a]{Nagata}, the field $k (\xi)$ is $C_r$.
Here we say that a field $K$ is $C_i$ if every homogeneous form in $K$ in $n$ variables and of degree $d$ with $n > d^i$ has a nontrivial zero in $K$.
Moreover we have $n-r+1 > 2^r$ by the definition of $r$.
It follows from the definition of $C_r$-field that $E_{\xi}$ has a $k (\xi)$-rational point.
This shows that $E_{\xi}$ is rational over $k (\xi)$ and $E_{\xi}$ is universally $\CH_0$-trivial.
By \cite[Proposition 1.8]{CTP}, we conclude that $\psi$ is a universally $\CH_0$-trivial resolution, as desired.
Therefore the proof of Theorem \ref{mainthm2} is completed.

\end{document}